\documentclass[12pt,reqno]{article}

\usepackage[usenames]{color}
\usepackage{amssymb}
\usepackage{amsmath}
\usepackage{amsthm}
\usepackage{amsfonts}
\usepackage{amscd}
\usepackage{graphicx}

\usepackage[colorlinks=true,
linkcolor=webgreen,
filecolor=webbrown,
citecolor=webgreen]{hyperref}

\definecolor{webgreen}{rgb}{0,.5,0}
\definecolor{webbrown}{rgb}{.6,0,0}

\usepackage{color}
\usepackage{fullpage}
\usepackage{float}

\usepackage{graphics}
\usepackage{latexsym}
\usepackage{epsf}
\usepackage{breakurl}

\begin{document}
	
	
	\theoremstyle{plain}
	\newtheorem{theorem}{Theorem}
	\newtheorem{corollary}[theorem]{Corollary}
	\newtheorem{lemma}[theorem]{Lemma}
	\newtheorem{proposition}[theorem]{Proposition}
	
	\theoremstyle{definition}
	\newtheorem{definition}[theorem]{Definition}
	\newtheorem{example}[theorem]{Example}
	\newtheorem{conjecture}[theorem]{Conjecture}
	
	\theoremstyle{remark}
	\newtheorem{remark}[theorem]{Remark}
	
	\begin{center}
		\vskip 1cm{\LARGE\bf A Note on Sumsets and Restricted Sumsets
			\vskip 1cm}
		\large
		Jagannath Bhanja\\
		Harish-Chandra Research Institute\\ 
		Chhatnag Road, Jhunsi\\ 
		Prayagraj 211019\\ 
		India \\
		\href{mailto:jagannathbhanja@hri.res.in}{\tt jagannathbhanja@hri.res.in} \\
	\end{center}
	\vskip .2 in

	\begin{abstract} 
		In this note we find the optimal lower bound for the size of the sumsets $HA$ and $H\,\hat{}A$ over finite sets $H, A$ of nonnegative integers, where $HA = \bigcup_{h\in H} hA$ and $H\,\hat{}A = \bigcup_{h\in H} h\,\hat{}A$. We also find the underlying algebraic structure of the sets $A$ and $H$ for which the size of the sumsets $HA$ and $H\,\hat{}A$ is minimum.  
	\end{abstract}

	\section{Introduction}\label{sec1}
	
	For a given finite set $A$ of integers and for a positive integer $h$, the {\em sumset} $hA$ and the {\em restricted sumset} $h\,\hat{}A$ are fundamental objects in the field of \emph{additive number theory}. The sumset $hA$ is the set of integers that can be written as the sum of $h$ elements of $A$, whereas the sumset $h\,\hat{}A$ is the set of integers that can be written as the sum of $h$ 
	{\it pairwise distinct\/} elements of $A$. In this regard, two of the important problems in additive number theory are to find the best possible lower bounds for the size of the sumsets $hA$ and $h\,\hat{}A$, and to find the structure of the finite set $A$ for which the sumsets $hA$ and $h\,\hat{}A$ contain the
	minimum number of elements. These two problems have been well established in the group of integers \cite{nathu, nathanson}.
	
	\begin{theorem} {\cite[Theorem 1.4, Theorem 1.6]{nathanson}} \label{theoremI}
		Let $A$ be a finite set of $k$ integers. Let $h$ be a positive integer. Then
		\begin{equation*}
			|hA|\geq h(k-1)+1.
		\end{equation*}
		Moreover, if this lower bound is exact with $h\geq 2$, then $A$ is an arithmetic progression.
	\end{theorem}
	
	\begin{theorem} {\cite[Theorem 1, Theorem 2]{nathu}} \label{theoremII}
		Let $A$ be a finite set of $k$ integers. Let $h\leq k$ be a positive integer. Then
		\begin{equation*}
			|h\,\hat{}A|\geq h(k-h)+1.
		\end{equation*}
		Moreover, if this lower bound is exact with $k\geq 5$ and $2\leq h\leq k-2$, then $A$ is an arithmetic progression.
	\end{theorem}

	Now let $H$ be a given finite set of nonnegative integers. Define the sumset \cite[p.\ 175]{bajnok}
	\[HA := \bigcup_{h\in H} hA,\]
	and the restricted sumset
	\[H\,\hat{}A := \bigcup_{h\in H} h\,\hat{}A.\]
	Here we are assuming that $0A=0\,\hat{}A=\{0\}$.

	For a set $A$ and for an integer $c$, we let $c \cdot A = \{ca: a \in A\}$. For integers $a, b$ with $a \leq b$, we also let $[a, b]=\{a,a+1,\ldots,b\}$.

	The case $H=[0, r]$ is more interesting and has been studied before (for recent papers, see \cite{bajnok, balandraud, bhanja}). Bajnok \cite{bajnok} defined the sumsets for an arbitrary finite set $H$ of nonnegative integers, and asked to 
	study similar problems like the sumsets $hA$ and $h\,\hat{}A$ over finite abelian groups. That is, find the optimal lower bound for the size of the sumsets $HA$ and $H\,\hat{}A$, and the structure of the sets $H$, $A$ for which the sumsets $HA$, $H\,\hat{}A$ contain the minimum number of elements.

	In this note, we study these two problems for the sumset $HA$ in Section \ref{sec2}, and the sumset $H\,\hat{}A$ in Section \ref{sec3}, for finite sets $A$ of nonnegative integers (or nonpositive integers) and $H$ of nonnegative integers. We consider two separate cases, namely
	\begin{itemize}
		\item[(i)] the set $A$ consists of positive integers and
		\item[(ii)] the set $A$ consists of nonnegative integers.
	\end{itemize}
	The cases
	\begin{itemize}
		\item[(iii)] the set $A$ consists of negative integers and
		\item[(iv)] the set $A$ consists of nonpositive integers,
	\end{itemize}
	follow from the cases (i) and (ii), respectively, as $H(c \cdot A) = c \cdot HA$ and $H\,\hat{} (c \cdot A) = c \cdot H\,\hat{}A$ for arbitrary integers $c$. As consequences of our results we obtain some recent results in this direction.

	In Section \ref{sec2} and Section \ref{sec3}, we use the following notation:
	for a set $S=\{s_{1}, s_{2}, \ldots, s_{k-1}, s_{k}\}$ with $k \geq 2$ and $s_{1}< s_{2}< \cdots< s_{k-1}< s_{k}$, we write $\min(S)=s_{1}$, $\min\nolimits_{+}(S)=s_{2}$, $\max(S)=s_{k}$, and $\max\nolimits_{-}(S)=s_{k-1}$.

	\section{Regular sumset}\label{sec2}
	
	\begin{theorem}\label{sumset-thm}
		Let $A$ be a set of $k$ positive integers. Let $H$ be a set of $r$ positive integers with $\max(H)=h_{r}$. Then
		\begin{equation}\label{sumset-eqn1}
			|HA| \geq h_{r}(k-1)+r.
		\end{equation}
		This lower bound is optimal.
	\end{theorem}
	
	\begin{proof}
		Let $A=\{a_{1},a_{2},\ldots,a_{k}\}$ and $H=\{h_{1},h_{2},\ldots,h_{r}\}$, where $0<a_{1}<a_{2}<\cdots<a_{k}$ and $0<h_{1}<h_{2}<\cdots<h_{r}$. Set
		\begin{equation}\label{set1}
			S_{1} := h_{1}A
		\end{equation}
		and
		\begin{equation}\label{seti}
			S_{i} := (h_{i}-h_{i-1})A+h_{i-1}a_{k},
		\end{equation}
		for $i=2,3,\ldots,r$. Clearly $S_{i} \subseteq h_{i}A$ for $i=1,2,\ldots,r$, and $\max(S_{i}) < \min(S_{i+1})$ for $i=1,2,\ldots,r-1$. Therefore $S_{1}, S_{2}, \ldots, S_{r}$ are pairwise disjoint subsets of $HA$. Hence, by Theorem \ref{theoremI}, we have 
		\begin{align}\label{sumset-eqn2}
			|HA| &\geq \sum_{i=1}^{r} |S_{i}|\nonumber\\
			&= |h_{1}A|+\sum_{i=2}^{r} |S_{i}|\nonumber\\
			&\geq h_{1}(k-1)+1+\sum_{i=2}^{r} [(h_{i}-h_{i-1})(k-1)+1]\nonumber\\
			&= h_{r}(k-1)+r.
		\end{align}
		This proves (\ref{sumset-eqn1}).

		To see that the lower bound in (\ref{sumset-eqn1}) is optimal, let $A=[1, k]$ and $H=[1, r]$. Then $HA=[1, rk]$, and  hence $|HA|=rk$. This completes the proof of the theorem.
	\end{proof}
	
	\begin{remark}
		If $A$ contains nonnegative integers with $0 \in A$, then $HA=h_{r}A$, as $h_{i}A \subseteq h_{r}A$ for $i=1,2,\ldots,r-1$. Therefore, by Theorem \ref{theoremI}, we have $|HA| \geq h_{r}(k-1)+1$. Furthermore, this bound is optimal, and it can be seen by taking $A=[0, k-1]$ and $H=[1, r]$.
	\end{remark} 
	
	Now we prove the inverse result of Theorem \ref{sumset-thm}.
	
	\begin{theorem}\label{sumset-inverse-thm}
		Let $A$ be a set of $k\geq 2$ positive integers and $H$ be a set of $r \geq 2$ positive integers with $\max(H)=h_{r}$. If $|HA|=h_{r}(k-1)+r$, then $H$ is an arithmetic progression of difference $d$ and $A$ is an arithmetic progression of difference $d \cdot \min(A)$.
	\end{theorem}
	
	\begin{proof}
		Let $A=\{a_{1},a_{2},\ldots,a_{k}\}$ and $H=\{h_{1},h_{2},\ldots,h_{r}\}$, where $0<a_{1}<a_{2}<\cdots<a_{k}$ and $0<h_{1}<h_{2}<\cdots<h_{r}$. Let $|HA|=h_{r}(k-1)+r$. Then the sumset $HA$ contains precisely the elements of the sets $S_{i}$ for $i=1, \ldots, r$, which are defined in (\ref{set1}), (\ref{seti}).

		First, we show that $A$ is an arithmetic progression. Observe that the assumption $|HA|=h_{r}(k-1)+r$ together with (\ref{sumset-eqn2}) implies $|h_{1}A|=h_{1}(k-1)+1$. If $h_{1}>1$, then from Theorem \ref{theoremI}, it follows that the set $A$ is an arithmetic progression. So, let $h_{1}=1$. Then
		\[S_{1}=h_{1}A=A=\{a_{1},a_{2},\ldots,a_{k}\}.\]
		Set \[S:=\{a_{1}, h_{2}a_{1}, (h_{2}-1)a_{1}+a_{2}, \ldots, (h_{2}-1)a_{1}+a_{k-1}\}.\]
		Clearly $S \subseteq HA$ and $\max(S)=(h_{2}-1)a_{1}+a_{k-1}<(h_{2}-1)a_{1}+a_{k}=\min(S_{2})$. Thus $S=S_{1}$. In other words, $(h_{2}-1)a_{1}+a_{i-1}=a_{i}$ for $i=2,3,\ldots,k$. Equivalently, $a_{i}-a_{i-1}=(h_{2}-1)a_{1}$ for $i=2,3,\ldots,k$. Hence, $A$ is an arithmetic progression.

		Next we show that $H$ is an arithmetic progression. For $i=1,2, \ldots,r-1$, consider the integers $(h_{i+1}-h_{i})a_{1}+a_{k-1}+(h_{i}-1)a_{k}$. Clearly
		\begin{multline*}
			\max\nolimits_{-}(S_{i})=a_{k-1}+(h_{i}-1)a_{k}<(h_{i+1}-h_{i})a_{1}+a_{k-1}+(h_{i}-1)a_{k}\\
			<(h_{i+1}-h_{i})a_{1}+h_{i}a_{k}=\min(S_{i+1}).
		\end{multline*}
		But we already have
		\begin{align*}
			\max\nolimits_{-}(S_{i})=a_{k-1}+(h_{i}-1)a_{k}<h_{i}a_{k}=\max(S_{i})<(h_{i+1}-h_{i})a_{1}+h_{i}a_{k}=\min(S_{i+1}).
		\end{align*}		
		Thus
		\[(h_{i+1}-h_{i})a_{1}+a_{k-1}+(h_{i}-1)a_{k}=h_{i}a_{k} \ \text{for} \ i=1,2, \ldots,r-1.\]
		This implies
		\begin{equation}\label{sumset-eqn3}
			a_{k}-a_{k-1}=(h_{i+1}-h_{i})a_{1} \ \text{for} \ i=1,2, \ldots,r-1.
		\end{equation}
		Therefore
		\[h_{2}-h_{1} = h_{3}-h_{2} = \cdots = h_{r}-h_{r-1},\]
		and hence the set $H$ is an arithmetic progression. Furthermore, by (\ref{sumset-eqn3}), the set $A$ is an arithmetic progression of difference $(h_{i+1}-h_{i})a_{1}$. This completes the proof of the theorem.
	\end{proof}

	\section{Restricted sumset}\label{sec3}
	
	\begin{theorem}\label{restricted-sumset-thm}
		Let $A$ be a set of $k$ positive integers and $H=\{h_{1}, h_{2}, \ldots, h_{r}\}$ be a set of positive integers with $h_{1}<h_{2}<\cdots<h_{r} \leq k$. Set $h_{0}=0$. Then
		\begin{equation}\label{restricted-eqn1}
			|H\,\hat{}A| \geq \sum_{i=1}^{r} (h_{i}-h_{i-1}) (k-h_{i})+r.
		\end{equation}
		This lower bound is optimal.
	\end{theorem}
	
	\begin{proof}
		Let $A=\{a_{1},a_{2},\ldots,a_{k}\}$, where $a_{1}<a_{2}<\cdots<a_{k}$. Set
		\begin{equation}\label{reset1}
			S_{1} := h_{1}\,\hat{}A
		\end{equation}
		and
		\begin{equation}\label{reseti}
			S_{i} := (h_{i}-h_{i-1})\,\hat{}A_{i}+\max(h_{i-1}\,\hat{}A),
		\end{equation}
		for $i=2,3,\ldots,r$, where $A_{i}=\{a_{1},a_{2},\ldots,a_{k-h_{i-1}}\}$. Clearly $S_{i} \subseteq h_{i}\,\hat{}A$ for $i=1,2,\ldots,r$, and $\max(S_{i}) < \min(S_{i+1})$ for $i=1,2,\ldots,r-1$. Therefore $S_{1}, S_{2}, \ldots, S_{r}$ are pairwise disjoint subsets of $H\,\hat{}A$. Hence, by Theorem \ref{theoremII}, we have
		\begin{align}\label{restricted-eqn2}
			|H\,\hat{}A| &\geq \sum_{i=1}^{r} |S_{i}|\nonumber\\
			&= |h_{1}\,\hat{}A|+\sum_{i=2}^{r} |S_{i}|\nonumber\\
			&\geq h_{1}(k-h_{1})+1+\sum_{i=2}^{r} [(h_{i}-h_{i-1})(k-h_{i})+1]\nonumber\\
			&= \sum_{i=1}^{r} (h_{i}-h_{i-1})(k-h_{i})+r.
		\end{align}
		This proves (\ref{restricted-eqn1}).

		Next to see that the lower bound in (\ref{restricted-eqn1}) is optimal, let $A=[1, k]$ and $H=[1, r]$ with $r \leq k$. Then $H\,\hat{}A \subseteq [1, k+(k-1)+\cdots+(k-r+1)]$. Therefore  $|H\,\hat{}A| \leq rk-\frac{r(r-1)}{2}$. This together with (\ref{restricted-eqn1}) implies $|H\,\hat{}A|=rk-\frac{r(r-1)}{2}$, and hence completes the proof of the theorem.
	\end{proof}
	
	As a consequence of Theorem \ref{restricted-sumset-thm}, we obtain the following corollary.
	
	\begin{corollary}\label{restricted-sumset-nonve-cor}
		Let $A$ be a set of $k$ nonnegative integers with $0 \in A$. Let $H=\{h_{1}, h_{2}, \ldots, h_{r}\}$ be a set of positive integers with $h_{1}<h_{2}<\cdots<h_{r} \leq k-1$. Set $h_{0}=0$. Then
		\begin{equation}\label{restricted-eqn3}
			|H\,\hat{}A| \geq \sum_{i=1}^{r} (h_{i}-h_{i-1}) (k-h_{i}-1)+h_{1}+r.
		\end{equation}
		This lower bound is optimal.
	\end{corollary}
	
	\begin{proof}
		Let $A=\{0, a_{1}, a_{2}, \ldots, a_{k-1}\}$, where $0<a_{1}<a_{2}<\cdots<a_{k-1}$. Set $A^{\prime}=A\setminus \{0\}$. For $i=1,2,\ldots,h_{1}$, let 
		\[ s_{i} = \sum_{j=1,\, j \neq h_{1}-i+1}^{h_{1}} a_{j}. \]
		Then it is easy to see that $\{0\} \cup H\,\hat{}A^{\prime}  \subseteq H\,\hat{}A$ if $h_{1} = 1$ and $\{s_{1}, s_{2}, \ldots, s_{h_{1}}\} \cup H\,\hat{}A^{\prime} \subseteq H\,\hat{}A$ if $h_{1}> 1$, where $s_{1} < s_{2} < \cdots < s_{h_{1}} < \min(H\,\hat{}A^{\prime})$. So, by Theorem \ref{restricted-sumset-thm}, we get
		\begin{equation}\label{corrected-eqn}
			|H\,\hat{}A| 
			\geq |H\,\hat{}A^{\prime}|+h_{1} 
			\geq \sum_{i=1}^{r} (h_{i}-h_{i-1}) (k-h_{i}-1)+h_{1}+r.
		\end{equation}
		
		Furthermore, the optimality of the lower bound in (\ref{restricted-eqn3}) can be verified by taking $A=[0, k-1]$ and $H=[1, r]$, where $k, r$ are positive integers with $r \leq k-1$.
	\end{proof}

	The following result (which has recently been proved) is a particular case of Theorem \ref{restricted-sumset-thm} and Corollary \ref{restricted-sumset-nonve-cor}.
	
	\begin{corollary} {\cite[Theorem 2.1, Corollary 2.1]{bhanja}}
		Let $A$ be a set of $k$ nonnegative integers and $H=[0, r]$ with $r \leq k$. If $0 \notin A$, then
		\begin{equation*}
			|H\,\hat{}A| \geq rk-\frac{r(r-1)}{2}+1.
		\end{equation*}
		If $0 \in A$ and $r \leq k-1$, then
		\begin{equation*}
			|H\,\hat{}A| \geq rk-\frac{r(r+1)}{2}+1.
		\end{equation*}
		These lower bounds are optimal. 	
	\end{corollary}

	Now we prove the inverse theorem of Theorem \ref{restricted-sumset-thm}.
	
	\begin{theorem}\label{restricted-sumset-inverse-thm}
		Let $A$ be a set of $k\geq 6$ positive integers. Let $H=\{h_{1}, h_{2}, \ldots, h_{r}\}$ be a set of $r\geq 2$ positive integers with $h_{1}<h_{2}<\cdots<h_{r}\leq k-1$. Set $h_{0}=0$. If
		\[|H\,\hat{}A|=\sum_{i=1}^{r} (h_{i}-h_{i-1}) (k-h_{i})+r,\]
		then $H=h_{1}+[0, r-1]$ and $A = \min(A) \cdot [1, k]$.
	\end{theorem}
	
	\begin{proof}
		Let $A=\{a_{1},a_{2},\ldots,a_{k}\}$, where $0<a_{1}<a_{2}<\cdots<a_{k}$. Let $|H\,\hat{}A|=\sum_{i=1}^{r} (h_{i}-h_{i-1}) (k-h_{i})+r$. Then the sumset $H\,\hat{}A$ contains precisely the elements of the sets $S_{i}$ for $i=1, \ldots, r$, which are defined in (\ref{reset1}), (\ref{reseti}).

		First, we show that $A$ is an arithmetic progression. Since $|H\,\hat{}A| = \sum_{i=1}^{r} (h_{i}-h_{i-1}) (k-h_{i})+r$, from (\ref{restricted-eqn2}), it follows that $|h_{1}\,\hat{}A| = h_{1}(k-h_{1})+1$. If $h_{1}\geq 2$, then by Theorem \ref{theoremII}, the set $A$ is an arithmetic progression. Therefore, let $h_{1}=1$.
		
		If $h_{2}\geq 3$, then $h_{2}-h_{1}\geq 2$. By (\ref{restricted-eqn2}), we get $|S_{2}|=|(h_{2}-h_{1})\,\hat{}A_{2}|=(h_{2}-h_{1})(k-h_{2})+1$, where $A_{2}=\{a_{1},a_{2},\ldots,a_{k-1}\}$. Therefore, by Theorem \ref{theoremII}, the set $A_{2}$ is an arithmetic progression. To show that $A$ is an arithmetic progression, it is left to show that $a_{k}-a_{k-1}=a_{k-1}-a_{k-2}$. Consider the following integers:
		\begin{multline*}
			a_{k-2} < a_{1}+\cdots+a_{h_{2}-1}+a_{k-2} < a_{1}+\cdots+a_{h_{2}-1}+a_{k-1} < a_{1}+\cdots+a_{h_{2}-1}+a_{k}\\
			= \min(S_{2}).
		\end{multline*}	
		But we already have
		\[a_{k-2}<a_{k-1}<a_{k}<a_{1}+\cdots+a_{h_{2}-1}+a_{k}=\min(S_{2}),\]
		where $\{a_{k-2}, a_{k-1}, a_{k}\} \subseteq S_{1}$.
		Thus
		\[a_{1}+\cdots+a_{h_{2}-1}+a_{k-2}=a_{k-1} \ \text{and} \ a_{1}+\cdots+a_{h_{2}-1}+a_{k-1}=a_{k}.\]
		This implies
		\[a_{k}-a_{k-1}=a_{1}+\cdots+a_{h_{2}-1}=a_{k-1}-a_{k-2},\]
		and we are done.

		Now let $h_{2}=2$; i.e., $h_{2}-h_{1}=1$. Set $T_{1} := \{a_{1}, a_{2}, a_{1}+a_{2}, a_{1}+a_{3}, \ldots, a_{1}+a_{k-1}\}$.
		Clearly $T_{1} \subseteq h_{1}\,\hat{}A \cup h_{2}\,\hat{}A \subseteq H\,\hat{}A$ and $\max(T_{1})=a_{1}+a_{k-1}<a_{1}+a_{k}=\min(S_{2})$. Therefore $T_{1}=S_{1}$. That is 
		\[\{a_{1}, a_{2}, a_{1}+a_{2}, a_{1}+a_{3}, \ldots, a_{1}+a_{k-1}\}=\{a_{1}, a_{2}, a_{3}, \ldots, a_{k}\}.\]
		Thus $a_{i}=a_{1}+a_{i-1}$ for $i=3,4,\ldots, k$. Equivalently, $a_{i}-a_{i-1}=a_{1}$ for $i=3,4,\ldots, k$. To show that $A$ is an arithmetic progression it is enough to show $a_{2}-a_{1}=a_{k}-a_{k-1}$. Consider the integer $a_{2}+a_{k-1}$. Since $\max(S_{1})=\max(T_{1})=a_{1}+a_{k-1}<a_{2}+a_{k-1}<a_{2}+a_{k}=\min\nolimits^{+} (S_{2})$ and $\max(S_{1})=\max(T_{1})=a_{1}+a_{k-1}<a_{1}+a_{k}=\min(S_{2})<a_{2}+a_{k}=\min\nolimits^{+} (S_{2})$, we must have $a_{2}+a_{k-1}=a_{1}+a_{k}$. This proves $A$ is an arithmetic progression.

		Next we show that $H$ is an arithmetic progression. For $i=1,2, \ldots,r-1$, consider the following integers:
		\begin{multline*}
			\max\nolimits_{-} (S_{i})
			=a_{k-h_{i}}+a_{k-h_{i}+2}+\cdots+a_{k} < a_{k-h_{i}+1}+\cdots+a_{k}=\max(S_{i}) \\
			< a_{1}+\cdots+a_{h_{i+1}-h_{i}}+a_{k-h_{i}+1}+\cdots+a_{k}=\min(S_{i+1})
		\end{multline*}	
		and
		\begin{multline*}
			\max\nolimits_{-} (S_{i})
			=a_{k-h_{i}}+a_{k-h_{i}+2}+\cdots+a_{k} < a_{1}+\cdots+a_{h_{i+1}-h_{i}}+a_{k-h_{i}}+a_{k-h_{i}+2}+\cdots+a_{k}\\
			< a_{1}+\cdots+a_{h_{i+1}-h_{i}}+a_{k-h_{i}+1}+\cdots+a_{k}=\min(S_{i+1}).
		\end{multline*}
		Therefore \[a_{k-h_{i}+1}+\cdots+a_{k}=a_{1}+\cdots+a_{h_{i+1}-h_{i}}+a_{k-h_{i}}+a_{k-h_{i}+2}+\cdots+a_{k}\]
		for $i=1,2, \ldots, r-1$. This implies
		\begin{equation*}\label{equation3.2}
			a_{k-h_{i}+1}-a_{k-h_{i}}=a_{1}+\cdots+a_{h_{i+1}-h_{i}} \ \text{for} \ i=1,2, \ldots,r-1.
		\end{equation*}
		Since $A$ is an arithmetic progression, the difference between any two consecutive elements in $A$ is same. Therefore \[a_{2}-a_{1}=a_{k-h_{i}+1}-a_{k-h_{i}}=a_{1}+a_{2}+\cdots+a_{h_{i+1}-h_{i}} \ \text{for} \ i=1,2, \ldots,r-1.\] 
		This holds, only if $h_{i+1}-h_{i}=1$ for $i=1,2, \ldots,r-1$. Hence, $H=h_{1}+[0, r-1]$ and $A=a_{1} \cdot [1, k]$. This completes the proof of the theorem.
		
	\end{proof}

	\begin{corollary}\label{restricted-sumset-inverse-cor}
		Let $A$ be a set of $k\geq 7$ nonnegative integers with $0 \in A$. Let $H=\{h_{1}, h_{2}, \ldots, h_{r}\}$ be a set of $r\geq 2$ positive integers with $h_{1}<h_{2}<\cdots<h_{r}\leq k-2$. Set $h_{0}=0$. If
		\[|H\,\hat{}A|=\sum_{i=1}^{r} (h_{i}-h_{i-1}) (k-h_{i}-1)+h_{1}+r,\]
		then $H=h_{1}+[0, r-1]$ and 
		$A=\min(A\setminus \{0\}) \cdot [0, k-1]$.
	\end{corollary}

		\begin{proof}
		Let $A=\{0, a_{1}, a_{2}, \ldots, a_{k-1}\}$, where $0<a_{1}<a_{2}<\cdots<a_{k-1}$. Set $A^{\prime} = A \setminus \{0\}$. The equality $|H\,\hat{}A| = \sum_{i=1}^{r} (h_{i}-h_{i-1}) (k-h_{i}-1)+h_{1}+r$ together with (\ref{corrected-eqn}) implies $|H\,\hat{}A^{\prime}| = \sum_{i=1}^{r} (h_{i}-h_{i-1}) (k-1-h_{i})+r$. By applying Theorem \ref{restricted-sumset-inverse-thm} on $H$ and $A^{\prime}$, we obtain $H=h_{1}+[0, r-1]$ and $A^{\prime} = \min(A^{\prime}) \cdot [1, k-1]$. Hence, $H=h_{1}+[0, r-1]$ and $A=\min(A^{\prime}) \cdot [0, k-1]$. This completes the proof of the corollary.
	\end{proof}
		
		The following inverse result (which has recently been proved) is a particular case of Theorem \ref{restricted-sumset-inverse-thm} and Corollary \ref{restricted-sumset-inverse-cor}.
		
	\begin{corollary} {\cite[Theorem 2.2, Corollary 2.3]{bhanja}}
		Let $A$ be a set of $k\geq 7$ nonnegative integers and $H=[0, r]$ with $2 \leq  r \leq k-1$. If $0 \notin A$ and $|H\,\hat{}A| = rk-\frac{r(r-1)}{2}+1$, then $A = d \cdot [1, k]$ for some positive integer $d$.
		
		If $0 \in A$, $r \leq k-2$, and $|H\,\hat{}A| = rk-\frac{r(r+1)}{2}+1$, then $A = d \cdot [0, k-1]$ for some positive integer $d$.
	\end{corollary}

	\section{Acknowledgement}
	We would like to thank Professor R. Thangadurai for carefully going
	through this manuscript. We also acknowledge the support of the PDF
	scheme of Harish-Chandra Research Institute, Prayagraj.

	\bigskip
	\hrule
	\bigskip
	
	\noindent 2020 {\it Mathematics Subject Classification}:
	Primary 11P70; Secondary 11B75, 11B13.
	
	\noindent \emph{Keywords:} sumset, restricted sumset.

\end{document}